%% file: submitted-version.tex
\documentclass[preprint,5p,times,twocolumn]{elsarticle}
\usepackage{amssymb}

\usepackage[usenames,dvipsnames]{xcolor}
\usepackage[utf8]{inputenc}
\usepackage[english]{babel}
\usepackage{verbatim}
\usepackage{amsmath,amssymb,amsthm,mathrsfs, mathtools}
\usepackage{stmaryrd}
\usepackage{dsfont}
\usepackage{textcomp}
\usepackage[mathcal]{euscript}
\usepackage{graphicx}
\usepackage{epsf}
\usepackage[colorlinks=true,urlcolor=blue,linkcolor=blue]{hyperref}
\usepackage{natbib}
\usepackage{geometry}
\usepackage{todonotes}
\usepackage{pbox}
\usepackage{pdflscape} 
\usepackage{lscape} 
\usepackage[linesnumbered, lined, boxed, commentsnumbered]{algorithm2e}
\usepackage{amsfonts}

\usepackage{subfig}
\usepackage{caption}

\newtheorem{theorem}{Theorem}

\newtheorem{example}[theorem]{Example}

\newtheorem{remark}[theorem]{Remark}

\theoremstyle{definition}
\newtheorem{mydef}{Definition}

\theoremstyle{plain}
\newtheorem{theo}{Theorem}

\newtheorem{cor}[theo]{Corollary}
\newtheorem{prop}[theo]{Proposition}

\theoremstyle{remark}

\newcommand{\1}{\mathds{1}}

\input{commandes.tex}

\newcommand{\RNSP}[1]{\textrm{RnSp}\np{#1}}
\newcommand{\RASP}[1]{\textrm{RaSp}\np{#1}}
\newcommand{\RNEQ}[1]{\textrm{RnEq}\np{#1}}
\newcommand{\RAEQ}[1]{\textrm{RaEq}\np{#1}}
\newcommand{\RAAD}[1]{\textrm{RaEq-AD}\np{#1}}

\makeatletter
\def\amsbb{\use@mathgroup \M@U \symAMSb}
\makeatother

\usepackage{bbold}

\usepackage{tikz}

\usepackage{placeins}

\usetikzlibrary{arrows,automata}
\usetikzlibrary{decorations.text}

\journal{Operations Research Letters}

\begin{document}

\begin{frontmatter}

\title{On risk averse competitive equilibrium}

\author[label1,label3]{Henri Gerard}

\address[label1]{Université Paris-Est, CERMICS (ENPC), F-77455 Marne-la-Vallée, France}
\address[label2]{Electric Power Optimization Centre, The University of Auckland, New Zealand}
\address[label3]{Université Paris-Est, Labex Bézout, F-77455 Marne-la-Vallée, France}

\author[label1]{Vincent Leclere}

\author[label2]{Andy Philpott}

\begin{abstract}

We discuss risked competitive partial equilibrium in a setting in which
agents are endowed with coherent risk measures. In contrast to social
planning models, we show by example that risked equilibria are not unique,
even when agents' objective functions are strictly concave. We also show that standard computational methods find only a subset of the equilibria, even with multiple starting points.
\end{abstract}


\end{frontmatter}

\section{Introduction}

Most industrialised regions of the world have over the last thirty years established wholesale electricity markets that take the form of an auction that matches supply and demand. The exact form of these auction mechanisms vary by jurisdiction, but they typically require offers of energy from suppliers at costs they are willing to supply, and clear a market by dispatching these offers in order of increasing cost. Day-ahead markets such as those implemented in many North American electricity systems, seek to arrange supply well in advance of its demand, so that thermal units can be prepared in time. Since the demand cannot be predicted with absolute certainty, day-ahead markets must be accompanied by a separate balancing market to deal with the variation in load and generator availability in real time. These are often called {\em two-settlement} markets.
The market mechanisms are designed to be as efficient as possible in the sense that they should aim to maximize the total welfare of producers and consumers.

In response to pressure to reduce $CO_{2}$ emissions and increase the penetration of renewables, electricity pool markets are procuring increasing amounts of electricity from 
 intermittent sources such as wind and solar. If probability distributions for intermittent supply are known for these systems then it makes sense to maximize the expected total welfare of producers and consumers in each dispatch. Then many repetitions of this will yield a long run total benefit that is maximized. Maximizing expected welfare can be modeled as a two-stage stochastic program. Methods for computing prices and single-settlement payment mechanisms for such a {\em stochastic market clearing} mechanism are described in a number of papers (see \citet{PZP}, \citet{WongEtAl} and \citet{ZPBB}). When evaluated using the assumed probability distribution on supply, stochastic market clearing can be shown to be more efficient than two-settlement systems.
 
 If agents in these systems are risk averse then one might also seek to maximize some risk-adjusted social welfare. In this setting the computation of prices and payments to the agents becomes more complicated. If agents use coherent risk measures then it is possible to define a complete market for risk in a precise sense. If the market is complete then a perfectly competitive partial equilibrium will also maximize risk-adjusted social welfare, i.e. it is efficient. On the other hand if the market for risk is not complete, then perfectly competitive partial equilibrium can be inefficient. This has been explored in a number of papers (see e.g. \citet{de2017investment}, \citet{ehrenmann2011generation} and \citet{ralph2015risk}).
 
 In this paper we study a class of stochastic dispatch and pricing mechanisms under the assumption that agents will attempt to maximize their risk-adjusted welfare at these prices. Agents have coherent risk measures and are assumed to behave as price takers in the energy and risk markets. We aim at enlightening some difficulties that arise when risk markets are not complete.  We describe a simple instance of a stochastic market that has three different equilibria. Two of these points are stable in the sense of \citet{samuelson1941stability} and are attractors of tatônnement algorithms. The third equilibrium is unstable, yet is the solution yielded by the well-known PATH solver in GAMS (See \citet{ferris2000complementarity}). Our example  illustrates the delicacy of seeking numerical solutions for equilibria in incomplete markets. Since these are used for justifying decisions, the nonuniqueness of solutions in this setting is undesirable.

The paper is laid out as follow. In Section~\ref{sec:statement} we present the equilibrium and optimization models we are going to study. In Section~\ref{sec:equivalences} we give links between equilibrium and optimization problems in the risk neutral and complete risk-averse cases. Finally, in Section~\ref{sec:multiple_eq} we showcase a simple example with multiple equilibria in the incomplete risk-averse case.

\subsection{Notation}
We use the following notation throughout the paper: $\nce{a;b}$ is the set of integers between $a$ and $b$ (included), random variables are denoted in bold,  $\omeg$ is a finite sample space, $\prbt$ is a probability distribution over $\omeg$, $\espe_{\prbt}$ is used to refer to expectation with respect to $\prbt$, $\mesr$ is used to refer to a risk measure. We denote by $x \perp y$ the complementarity condition $x^{T}y=0$.

\section{Statement of problem}
\label{sec:statement}
Consider a two time-step single-settlement market for one good. 
In a single-settlement market, the producer can arrange in advance for a production of $x$ at 
a marginal cost $cx$ as a first-step decision, and choose the value of a recourse variable $\mathbf{x}_{r}$ incurring an uncertain marginal cost $\mathbf{c}_{r}\mathbf{x}_{r}$. We assume that there are a finite number of scenarios $\omega \in \omeg$ determining the coefficient $\mathbf{c}_r\np{\omega}$.

The product is purchased in the second step by a consumer with a utility function 
$\mathbf{V}\np{\omega}\mathbf{y}\np{\omega}-\frac{1}{2}\mathbf{r}\np{\omega}\mathbf{y}^2\np{\omega}$. The consumer has no first-stage decision, and the amount purchased $\mathbf{y}\np{\omega}$ depends on the scenario.

\subsection{Social planner problem}
Decisions $x$, $\mathbf{x}_{r}\np{\omega}$ and $\mathbf{y}\np{\omega}$ can be made to maximize a
social objective.
We denote by
\begin{subequations}\label{eq_def_welfares}
    \begin{align}
        \va{W}_{p}\np{\omega} 
        &= 
        -\frac{1}{2}cx^{2}-\frac{1}{2}\mathbf{c}_{r}\np{\omega}\mathbf{x}_{r}\np{\omega}^{2}
        \eqfinv
        \intertext{the welfare of the producer, and by}
        \va{W}_c\np{\omega}
        &= 
        \mathbf{V}\np{\omega}\mathbf{y}\np{\omega}-\frac{1}{2}\mathbf{r}\np{\omega }\mathbf{y}\np{\omega}^{2}\eqfinv
    \end{align}
\end{subequations}    
the welfare of the consumer where both these expressions ignore the price paid for the good in scenario $\omega$. 
Then the welfare of the social planner can be defined by 
$\va{W}_{sp} = \va{W}_{p} + \va{W}_{c}$.

Optimization of the social objective requires us to aggregate the uncertain outcomes from the scenarios. This can be done by taking expectations with respect to an underlying probability measure $\prbt$ or using a more general risk measure.

\subsubsection{Risk neutral social planner problem}
Endow the set of scenario $\Omega$ with a probability $\prbt$, 
then a risk-neutral social planner might seek to maximize
the expected total social welfare under the constraint that supply equals demand. This problem is denoted by $\RNSP{\prbt}$ and reads
\begin{subequations}
    \begin{align}\label{eq_problem_RNSP}
        &\RNSP{\prbt}:
        \nonumber\\
        &\max_{x, \mathbf{x}_{r}, \mathbf{y}}
        \quad
        \nesp{\prbt}{\va{W}_{sp}}
        \eqfinv\\
        &\textrm{s.t. }
        \quad
        x + \mathbf{x}_{r}(\omega) \geq \mathbf{y}(\omega)
        \eqsepv
        \forall \omega \in \omeg
        \eqfinp
    \end{align}
\end{subequations}



\subsubsection{Risk averse social planner problem}
Choosing expectation $\espe_\PP$, assumes a risk-neutral point of view, where two random losses with same expectation but different variances are deemed equivalent. In practice a number of agents are risk averse. To model risk aversion we generally use a \emph{risk measure} $\FF$, that is a functional that associates to a random welfare its deterministic equivalent, i.e. the deterministic welfare deemed as equivalent to the random loss. 

A risk-averse planner solves a maximization problem $\RASP{\mesr}$ defined by
\begin{subequations}\label{eq_problem_RASP}
    \begin{align}
        &\RASP{\mesr}:&
        \nonumber\\
        &\max_{x, \mathbf{x}_{r}, \mathbf{y}} 
        \quad
        \nmes{}{\va{W}_{sp}}
        \eqfinv
        \\
        &\textrm{s.t. }
        \quad
        x + \mathbf{x}_{r}\np{\omega} \geq \mathbf{y}\np{\omega}
        \eqsepv
        \forall \omega \in \omeg
        \eqfinp
    \end{align}
\end{subequations}

A risk measure $\FF$ is said to be coherent (see \citet{artzner1999coherent}) if it satisfies four natural properties: monotonicity ( if $\va X \geq \va Y$ then $\FF[\va X] \geq \FF[\va y]$), concavity ($\FF$ is concave), translation-equivariance ($\FF[\va X+c]=\FF[\va X]+c$ with $c\in\RR$) and positive homogeneity ($\FF[\lambda \va X] = \lambda \FF[\va X]$, with $\lambda \geq 0$).
 By convex duality theory (see \citet{shapiroetal}), a lower-semicontinuous coherent risk measure can be written
    $\bmes{}{\va{Z}}   =
    \min_{\probelement \in \probset}
    \besp{\probelement}{\va{Z}}$,
where $\probset$ is a closed, convex, non-empty set of probability distributions over $\omeg$.
If $\probset$ is a polyhedron defined by $K$ extreme points $\np{\probelement_{k}}_{k \in \nce{1;K}}$, then the risk measure is denoted $\check{\FF}$ and said to be \emph{polyhedral}, with
 $   \check{\FF}[\va Z]
    =
    \min_{\probelement_{1},\dots ,\probelement_{K}}
    \besp{\probelement_{k}}{\va{Z}}$.

Problem $\RASP{\check{\FF}}$ can be written as follows
\begin{subequations}\label{eq_problem_risk_optimization}
    \begin{align}
        &\RASP{\check{\FF}}:
        \nonumber\\
        &\max_{\theta,x,\mathbf{x}_{r},\mathbf{y}} 
        \quad 
        \theta
        \\
        &\textrm{s.t.} 
        \qquad 
        \theta \leq
        \besp{\probelement_{k}}{\va{W}_{sp}}\eqsepv
        \forall k \in \nce{1;K}
        \eqfinv
        \label{eq_constraint_RASP}
        \\
        &\phantom{s.t.}
        \qquad
        x + \mathbf{x}_{r}\np{\omega} \geq \mathbf{y}\np{\omega}
        \eqsepv
        \forall \omega \in \omeg
        \eqfinp
    \end{align}
\end{subequations}
In what follows we assume that all risk measures are coherent.
\subsubsection{Remark on non linearity of risk averse objective function}
By linearity of expectation we have
$    \nesp{\prbt}{\va{W}_{sp}}
    =
    \nesp{\prbt}{\va{W}_{p}}
    +
    \nesp{\prbt}{\va{W}_{c}}$
hence the criterion of the social planner is natural, 
which is not the case anymore with risk-aversion. The social planner criterion could be either $\nmes{}{\va{W}_{sp}}$ or $\nmes{}{\va{W}_{p}} + \nmes{}{\va{W}_{c}}$.
Furthermore, by concavity and positive homogeneity, we have $\nmes{}{\va{W}_{p} + \va{W}_{c}} \geq \nmes{}{\va{W}_{p}}+\nmes{}{\va{W}_{c}}$.

\subsection{Equilibrium problem}
We now define a competitive partial equilibrium for our model. This competitive equilibrium can be risk neutral or risk averse. Definitions come from general equilibrium theory (See \citet{arrow1954existence} or \citet{uzawa1960walras}).

\subsubsection{Risk neutral equilibrium}
Given a probability $\prbt$ on $\omeg$, a \emph{risk-neutral equilibrium} $\RNEQ{\prbt}$ is a set of prices $\ba{\va{\pi}\np{\omega}\eqsepv \omega \in \omeg}$ such that there exists 
a solution to the system 
\begin{subequations}\label{eq_problem_RNEQ}
    \begin{align}
        &\RNEQ{\prbt}:
        \nonumber\\
            &\begin{aligned}\label{eq_problem_producer_RNEQ}
                \max_{x,\mathbf{x}_{r}}
                \quad
                \Besp{\prbt}{\va{W}_{p} + \va{\pi}\bp{x+\mathbf{x}_{r}}}
                \eqfinv
            \end{aligned}
            \\&
            \phantom{i}\begin{aligned}\label{eq_problem_consumer_RNEQ}
                \max_{\mathbf{y}}
                \quad
                \besp{\prbt}{\va{W}_{c} - \va{\pi}\mathbf{y}}
                \eqfinv
            \end{aligned}
            \\&
            \begin{aligned}\label{eq_problem_constraint_RNEQ}
                0 \leq
                x+\mathbf{x}_{r}\np{\omega} - \mathbf{y}\np{\omega}
                \perp
                \va{\price}\np{\omega}
                \geq 0
                \eqsepv
                \forall \omega \in \omeg
                \eqfinp
            \end{aligned}
    \end{align}
\end{subequations}

Here, the producer maximizes its expected profit~\eqref{eq_problem_producer_RNEQ}, the consumer maximizes its expected utility~\eqref{eq_problem_consumer_RNEQ} and the market clears with~\eqref{eq_problem_constraint_RNEQ} (which means that either prices are null or supply equals demand). As the consumer has no first stage decision, she can optimize each scenario independently and so problem~\eqref{eq_problem_consumer_RNEQ} can be replaced by
\begin{equation*}
    \max_{\mathbf{y}\np{\omega}}
    \quad
    \va{W}_{c}(\omega) - \va{\pi}\np{\omega}\mathbf{y}(\omega)
    \eqsepv
    \forall \omega \in \omeg
    \eqfinp
\end{equation*}

\subsubsection{Risk averse equilibrium}

Given two risk measures $\mesr_{p}$ and $\mesr_{c}$ over $\omeg$, \emph{a risk-averse equilibrium} $\RAEQ{\mesr_{p}, \mesr_{c}}$ is a set of prices $\ba{\va{\pi}\np{\omega}:\omega \in 
\omeg}$ such that there exists a solution to the following system
\begin{subequations}\label{eq_problem_RAEQ_general}
    \begin{align}
        &\RAEQ{\mesr_{p}, \mesr_{c}}:
        \nonumber \\
        &\begin{aligned}\label{eq_RAEQ_general_producer}
            \max_{x,\mathbf{x}_{r}}
            \quad
            \Bmes{p}{\va{W}_{p}+\va{\pi}\bp{x+\mathbf{x}_{r}}}
            \eqfinv
        \end{aligned}
        \\&\begin{aligned}\label{eq_RAEQ_general_consumer}
            \max_{\mathbf{y}}
            \quad
            \bmes{c}{\va{W}_{c}-\va{\pi}\mathbf{y}}
            \eqfinv
        \end{aligned}
        \\&\begin{aligned}\label{eq_RAEQ_general_complementarity}
            0 \leq
            x+\mathbf{x}_{r}\np{\omega} - \mathbf{y}\np{\omega}
            \perp
            \va{\price}\np{\omega}
            \geq 0
            \eqsepv
            \forall \omega \in \omeg
            \eqfinp
        \end{aligned}
    \end{align}
\end{subequations}

Since the coherent risk measure $\mesr_{c}$ of the consumer is monotonic, and noting that she has no first-stage decision, she can optimize scenario per scenario. Thus, she is 
\emph{insensitive to  risk} as any monotonic risk measure will lead to the same action (although not the same welfare). Since $\mesr_{p}$ is also monotonic, we can endow both 
agents with the same risk measure.
In that case, we denote problem~\eqref{eq_problem_RAEQ_general} by $\RAEQ{\mesr}$.

We now consider polyhedral risk measure $\check{\FF}$, using formulation~\eqref{eq_problem_risk_optimization}, the equilibrium problem~\eqref{eq_problem_RAEQ_general} reads  
\begin{subequations}\label{eq_problem_RAEQ_polyhedral}
    \begin{align}
        &\RAEQ{\check{\FF}}:
        \nonumber\\
        &\max_{\theta,x,\mathbf{x}_{r}}
        \quad
        \theta\\
        &\;\;\; \textrm{s.t.}
        \quad
        \theta \leq \besp{\probelement_{k}}{\va{W}_{p}+\va{\pi}\np{x+\mathbf{x}_{r}}}
        \eqsepv \forall k \in \nce{1;K}
        \eqfinv
        \nonumber\\
        &\max_{\mathbf{y}(\omega)}
        \quad
        \va{W}_{c}(\omega) - \va{\pi}\mathbf{y}(\omega)
        \eqsepv
        \forall \omega \in \omeg
        \eqfinv
        \\
        &0 \leq
        x+\mathbf{x}_{r}\np{\omega} - \mathbf{y}\np{\omega}
        \perp
        \va{\price}\np{\omega}
        \geq 0
        \eqsepv
        \forall \omega \in \omeg
        \eqfinp
    \end{align}
\end{subequations}

\subsection{Trading risk with Arrow-Debreu securities}
Until now, we have considered equilibrium problems in an incomplete market. Following the path of~\citet{PhilpottFerrisWets16}, we complete the market using Arrow-Debreu securities. 

\begin{mydef}
An \emph{Arrow-Debreu security} for node $\omega \in \omeg$ is a contract that charges a price $\mu\np{\omega}$ in the first stage, to receive a payment of $1$ in scenario $\omega$.
\end{mydef}

The consumer  now has a first-stage decision which is the number of contracts she buys, so the choice of the consumer risk measure $\mesr_{c}$ has now consequences. For convenience, 
this risk measure $\mesr_{c}$ is chosen to be the same as that of the producer $\mesr_{p}$ and will be denoted by $\mesr$. Unless stated otherwise, from now on we use polyhedral risk measures.

Denote $\mathbf{a}\np{\omega}$ (resp. $\mathbf{b}\np{\omega}$) the number of Arrow-Debreu securities bought by the producer (resp. the consumer). We denote by $\va{\mu}(\omega)$ the price of the Arrow-Debreu securities associated with scenario $\omega$. In this case the producer pays $\sum_{\omega \in \omeg} \va{\mu}\np{\omega}\mathbf{a}\np{\omega}$ in the first stage, in order to receive $\mathbf{a}\np{\omega}$ in scenario $\omega$. As $\mathbf{a}\np{\omega}+\mathbf{b}\np{\omega}$ represents excess demand, requiring that supply is greater than demand consists in requiring  $\mathbf{a}\np{\omega}+\mathbf{b}\np{\omega} \leq 0$.
Prices $\na{\va{\pi}\np{\omega}, \va \mu(\omega)}_{\omega\in\Omega}$ form a \emph{risk-trading
equilibrium} if there exists a solution to:

\begin{subequations}\label{eq_problem_RAEQ_AD}
    \begin{align}
            &\RAAD{\check{\FF}}:
            \nonumber\\
            &\max_{\theta,x,\mathbf{x}_{r},\mathbf{a}}
            \quad
            \theta - \sum_{\omega \in \omeg} \va{\mu}\np{\omega}\mathbf{a}\np{\omega}
            \\
            &\;\;\;\textrm{ s.t.}
            \quad
            \theta \leq \Besp{\probelement_{k}}{\va{W}_{p} + \va{\pi}\np{x+\mathbf{x}_{r}} + \mathbf{a}}
            \eqsepv \forall k \in \nce{1;K}
            \eqfinv
            \label{cst:lambda_k}
            \\
            &\max_{\phi,\mathbf{y},\mathbf{b}}
            \quad
            \phi - \sum_{\omega \in \omeg} \va{\mu}\np{\omega}\mathbf{b}\np{\omega}\\
            &\;\;\;\textrm{ s.t.}
            \quad
            \phi \leq \besp{\probelement_{k}}{\va{W}_{c} - \va{\pi}\mathbf{y} + \mathbf{b}}
            \eqsepv \forall k \in \nce{1;K}
            \eqfinv
            \label{cst:sigma_k}
            \\
            &0 \leq
            x+\mathbf{x}_{r}\np{\omega} - \mathbf{y}\np{\omega}
            \perp
            \va{\price}\np{\omega}
            \geq 0
            \eqsepv
            \forall \omega \in \omeg
            \eqfinv
            \label{cst:market_clear}
            \\
            &0 \leq
            -\mathbf{a}\np{\omega}-\mathbf{b}\np{\omega}
            \perp
            \va{\mu}\np{\omega}
            \geq 0
            \eqsepv
            \forall \omega \in \omeg
            \eqfinp  \label{cst:mu_clear}
    \end{align}
\end{subequations}

\section{Some equivalences between social planner problems and equilibrium problems}
\label{sec:equivalences}
We recall a trivial equivalence between problem $\RNSP{\prbt}$ and problem $\RNEQ{\prbt}$ before showing an equivalence between problem $\RASP{\check{\FF}}$ and problem $\RAAD{\check{\FF}}$.

\subsection{Equivalence in the risk neutral case}
\begin{prop}\label{prop_equivalence_RNSP_RNEQ}
    Let $\prbt$ be a probability measure over $\omeg$.
    The elements $x^{\opt}$, $\mathbf{x}_{r}^{\opt}$ and $\mathbf{y}^{\opt}$ are optimal 
    solutions to $\RNSP{\prbt}$ if and only if there exist equilibrium prices $\va{\price}^{\opt}$ for $\RNEQ{\prbt}$ with 
    associated optimal decisions $x^{\opt}$, $\mathbf{x}_{r}^{\opt}$ and $\mathbf{y}^{\opt}$.
\end{prop}

\begin{proof}
As the producer and the consumer optimize over different uncoupled variables, it is equivalent to optimize their objectives separately or jointly. Problem~\eqref{eq_problem_RNEQ} is thus equivalent to
\begin{subequations}
    \begin{align*}
            &\begin{aligned}
                \max_{x,\mathbf{x}_{r},\mathbf{y}} \qquad 
                \besp{\prbt} {\va{W}_{p} + \va{\pi}
                \np{x+\mathbf{x}_{r}}}
                +
                \besp{\prbt}{\va{W}_{c} - \va{\pi}\mathbf{y}} 
                \eqfinv
            \end{aligned}
            \\&\begin{aligned}
                0 \leq
                x+\mathbf{x}_{r}\np{\omega} - \mathbf{y}\np{\omega}
                \perp
                \va{\price}\np{\omega}
                \geq 0
                \eqsepv
                \forall \omega \in \omeg
                \eqfinv
            \end{aligned}
    \end{align*}
\end{subequations}
which by linearity of the expectation is equivalent to
\begin{subequations}
    \begin{align*}
            &\begin{aligned}
                \max_{x,\mathbf{x}_{r},\mathbf{y}} \qquad 
                \besp{\prbt}{
                    \va{W}_{sp} + \va{\pi}\np{x+\mathbf{x}_{r}-\mathbf{y}}
                }
                \eqfinv
            \end{aligned}
            \\&\begin{aligned}\label{eq_constraint_complementarity}
                0 \leq
                x+\mathbf{x}_{r}\np{\omega} - \mathbf{y}\np{\omega}
                \perp
                \price\np{\omega}
                \geq 0
                \eqsepv
                \forall \omega \in \omeg
                \eqfinp
            \end{aligned}
    \end{align*}
\end{subequations}
This is equivalent to the optimality conditions for problem~\eqref{eq_problem_RNSP}.
Convexity and linearity of constraints ends the proof.
\end{proof} 
\begin{cor}\label{cor_uniqueness_risk_neutral}
    If both the producer's and the consumer's criterion are strictly concave and if $\prbt$ charges all $\omega$, then $\RNSP{\prbt}$ admits a unique solution and $\RNEQ{\prbt}$ admits a unique equilibrium.
\end{cor}
\begin{proof}
    The probability distribution $\prbt$ charges all $\omega$. Then by strict concavity, $\RNSP{\prbt}$ has a unique solution. If $\RNEQ{\prbt}$ has two different solutions $\np{x^{1},\mathbf{x}_{r}^{1},\mathbf{y}^{1}}$ and $\np{x^{2},\mathbf{x}_{r}^{2},\mathbf{y}^{2}}$ with $\va{\price}^{1}$ and $\va{\price}^{2}$ respectively then, by Proposition~\ref{prop_equivalence_RNSP_RNEQ}, $x^{1} = x^{2}$, $\mathbf{x}_{r}^{1} = \mathbf{x}_{r}^{2}$, and $\mathbf{y}^{1} = \mathbf{y}^{2}$. Since~\eqref{eq_problem_consumer_RNEQ} implies $\va{\price}^{1}\np{\omega} = \mathbf{V}\np{\omega}-\mathbf{r}\np{\omega}\mathbf{y}^{1}\np{\omega}$, we have $\price^{1} = \price^{2}$ which gives the result.
\end{proof}

\subsection{Equivalence in the risk-averse case}
The following proposition is an extension of Theorem 7 of~\citet{ralph2015risk}, to a model with producers and consumers, in the special case of a finite number of scenarios with polyhedral risk measures.
\begin{prop}\label{prop_RASP_RAEQAD}
Let $\va{\price}$ and $\va{\mu}$ be equilibrium prices such that $\bp{x^{\opt},\mathbf{x}_{r}^{\opt},\mathbf{y}^{\opt},\mathbf{a},\mathbf{b},\theta ,\varphi }$ solves
$\RAAD{\check{\FF}}$. Then 

(i) $\va \mu$ is a probability measure, and $x^{\opt},\mathbf{x}_{r}^{\opt},\mathbf{y}^{\opt}$ solves 
the risk-neutral social
planning problem when evaluated using probability $\va{\mu}$, 
$\RNSP{\va \mu}$.

(ii) $x^{\opt},\mathbf{x}_{r}^{\opt},\mathbf{y}^{\opt}$ solves 
the risk-averse social planning problem, 
$\RAAD{\check{\FF}}$
with worst case measure $\va{\mu}$.
\end{prop}

\begin{proof}
(i) \ Each agent problem is convex with linear constraints. Hence the optimal solution satisfies for each problem the Karush-Kuhn-Tucker (KKT) conditions.
The Lagrangian of the producer problem reads
\begin{equation*}
    \cL_p = \theta - \sum_{\omega \in \Omega }\mathbf{\mu}\np{\omega}\mathbf{a}\np{\omega} 
    + \sum_k \lambda_k \Bp{
    \EE_{\QQ_k}\bc{\va W_p + \va \price (x+\mathbf{ x}_r) + \mathbf{a}} - \theta }
    \eqfinv
\end{equation*}
where $\lambda_k$ is the multiplier associated to constraint~\eqref{cst:lambda_k}.
Then, the KKT conditions imply that $\sum_k \lambda_k = 1$, and 
$\va \mu = \sum_{k} \lambda_k \QQ_k$. In particular, $\va \mu$ is a probability measure in $\cQ$.
Furthermore $(x^{\opt}, \mathbf{x}_r^{\opt})$ maximizes $\sum_{\omega \in \Omega} \va \mu(\omega)  \bp{\va W_{p}(\omega) - \price(\omega)(x+\mathbf{x}_r(\omega)) }$ which is the risk-neutral producer objective evaluated with measure $\va \mu$. 

Similarly, looking at the consumer problem with multiplier $\sigma_k$ associated to constraint~\eqref{cst:sigma_k}, we obtain $\sum_k \sigma_k = 1$ and $\va \mu = \sum_{k} \sigma_k \QQ_k$. Hence, the consumer maximizes her risk-neutral objective under the same probability $\va \mu$ as the producer.

Since by hypothesis the solutions satisfy~\eqref{cst:market_clear}
we have that $\bp{x^{\opt},\mathbf{x}_{r}^{\opt}\np{\omega},\mathbf{y}^{\opt}\np{\omega}}$ solves \RNSP{$\va \mu$}.

(ii) Observe that complementary slackness gives%
\begin{equation*}
    \begin{array}{l}
    \lambda _{k}\left(\EE_{\QQ_k}\bc{\va{W}_{p}^{\opt} +
    \va \pi\bp{x^{\opt}+\mathbf{x}_{r}^{\opt}}+\bar{\mathbf{a}})}-\bar{\theta}\right) =0\eqfinv\\ 
    \sigma _{k}\left( \EE_{\QQ_k}\bc{
    \va{W}_{c}^{\opt}-\va \pi
    \mathbf{y}^{\opt}+\bar{\mathbf{b}}} -\bar{\varphi}\right) =0
    \eqfinv
    \end{array}
\end{equation*}
where $\va{W}_{p}^{\opt}$ and $\va{W}_{c}^{\opt}$ are defined by~\eqref{eq_def_welfares} in terms of $x^{\opt}$, $\mathbf{x}_{r}^{\opt}$ and $\mathbf{y}^{\opt}$.
Summing over $k$, and leveraging~\eqref{cst:mu_clear} gives 
\begin{align}
    \bar{\theta}+\bar{\varphi} 
    &=
    \nesp{\va{\mu}}{
        \va{W}_{p}^{\opt}+\va{\price}\bp{x^{\opt}+\mathbf{x}_{r}^{\opt}}+\bar{\mathbf{a}}
    } 
    +\nesp{\va{\mu}}{
        \va{W}_{c}^{\opt}-\va{\price}\bar{y}+\bar{\mathbf{b}}
    }   
    \nonumber
    \eqfinv
    \\
    &=
    \nesp{\va{\mu}}{\va{W}_{p}^{\opt}+\va{W}_{c}^{\opt}} 
    \eqfinp
    \label{eqn:thetaphi}
\end{align}
However as%
\begin{eqnarray}
    \bar{\theta}+\bar{\varphi} &=&\min_{\va{\QQ} \in \probset}\nesp{\QQ}{ \va{W}_{p}^{\opt} + \va{\price}\bp{x^{\opt}+\mathbf{x}_{r}^{\opt}}+\bar{\mathbf{a}}}   \notag \\
    &&+\min_{\QQ' \in \probset}\nesp{\va{\QQ'}}{\va{W}_{c}^{\opt}-\va{\price}\mathbf{y}^{\opt}+\bar{\mathbf{b}}}   
    \eqfinv
    \notag \\
    &\leq &\min_{\QQ \in \probset}
    \nesp{\QQ}{\va{W}_{p}^{\opt}+\va{W}_{c}^{\opt}+\bar{\mathbf{a}}+\bar{\mathbf{b}}}
    \eqfinv
    \notag \\
    &\leq &\min_{\QQ \in \probset}\nesp{\QQ}{\va{W}_{p}^{\opt}+\va{W}_{c}^{\opt}}
    \eqfinp
    \label{eqn:thetaphi1}
\end{eqnarray}%
Combining (\ref{eqn:thetaphi}) and (\ref{eqn:thetaphi1}) and observing that $%
\mu \in \probset$, we have 
\begin{align}\label{eqn:mumin}
    \nesp{\va{\mu}}{\va{W}_{p}^{\opt}+\va{W}_{c}^{\opt}}  
    =
    \min_{\probelement \in \probset}\nesp{\probelement}{ \va{W}_{p}^{\opt}+\va{W}_{c}^{\opt}}.
\end{align}%
To complete the proof, consider any feasible $x,\mathbf{x}_{r}\np{\omega},\mathbf{y}\np{\omega}$. By
part (i) and $\va{\mu} \in \probset$, we have 
\begin{align*}
    \nesp{\va{\mu}}{\va{W}_{p}^{\opt}+\va{W}_{c}^{\opt}} &\geq \nesp{\va{\mu}}{\va{W}_{p}+\va{W}_{c}} \geq \min_{\QQ \in \probset}
    \nesp{\QQ}{ \va{W}_{p}+\va{W}_{c}}
    \eqfinv
\end{align*}
where $\va{W}_{p}$ and $\va{W}_{c}$ are defined by~\eqref{eq_def_welfares}.
Thus (\ref{eqn:mumin}) gives%
\begin{equation*}
    \min_{\QQ \in \probset}
    \nesp{\QQ}{\va{W}_{p}^{\opt}+\va{W}_{c}^{\opt}}
    \geq 
    \min_{\QQ \in \probset}
    \nesp{\QQ}{\va{W}_{p}+\va{W}_{c}}
    \eqfinp
\end{equation*}%
This shows that  
\begin{equation*}
    \bp{x^{\opt},\mathbf{x}_{r}^{\opt},\mathbf{y}^{\opt}}\in \argmax_{x,\mathbf{x}_{r},\mathbf{y}}
    \min_{\QQ \in \probset}
    \nesp{\QQ}{\va{W}_{p}+\va{W}_{c}}
    \eqfinv
\end{equation*}
as required.
\end{proof}

\begin{remark}
Note that an equilibrium of $\RAAD{\check{\FF}}$ consists of a price vector $\va \price$, giving one price per scenario, and a probability $\va \mu$ that is seen by both the producer and the consumer as a worst-case probability for the welfare plus trade evaluation. 
\end{remark}

\begin{remark}
    In Section~\ref{sec:multiple_eq} we give an example
    of three risked equilibrium without Arrow-Debreu securities, each
    corresponding to a risk-neutral equilibrium with different measure $\va{\mu}\np{\omega}$. However if Arrow-Debreu securities are included then two of these equilibria are no longer equilibria in a risk-averse setting. The
    risk-averse consumer, who without Arrow-Debreu securities had no mechanism
    to alter his outcomes will trade these securities to improve their
    risk-adjusted payoff.
\end{remark}

\begin{remark}
    Consider a set of prices $\va{\pi}$ that gives a risked equilibrium in which
    agent $i$ has payoff $\va{W}_{i}\np{\va{\pi}}$ and risked payoff 
    $\bmes{i}{\va{W}_{i}\np{\va{\pi}}}$. Suppose that there exists a
    probability measure $\probelement^{\ast}$ such that 
    $\bmes{i}{\va{W}_{i}\np{\va{\pi}}} = \Besp{\probelement^{\ast}}{\va{W}_{i}\np{\va{\pi}}}$. 
    Observe that this does not imply that choosing actions $x$ to
    maximize $\nesp{\probelement^{\ast}}{\va{W}_{i}\np{\va{\pi}}}$ will give 
    $\max_{x} \bmes{i}{\va{W}_{i}\np{\va{\pi}}}$. This is because $x^{\ast }$ solves 
    \begin{equation*}
        \max_{x}
        \bmes{i}{\va{W}_{i}\np{\va{\price}}}
        =
        \max_{x} \min_{\probelement \in \probset}
        \besp{\probelement}{\va{W}_{i}\np{\va{\price}}}
        \eqfinv
    \end{equation*}
    
    and not 
    \begin{equation*}
        \max_{x}
        \besp{\probelement^{\ast}}{f_{i}\np{x,\va{\pi}}}
        \eqfinv
    \end{equation*}
    since $\probelement^{\ast}$ depends on $x$.
\end{remark}

\begin{remark}
    Proposition~\ref{prop_RASP_RAEQAD} is easily extended to the case where the agents have different risk measures $\mesr_{p}$ and $\mesr_{c}$ with non-disjoint risk set. In this case, \eqref{eqn:thetaphi1} becomes 
    \begin{eqnarray}
        \bar{\theta}+\bar{\varphi} &=&\min_{\probelement_{p} \in \probset_{p}}\nesp{\QQ_{p}}{ \va{\price}\bp{x^{\opt}+\mathbf{x}_{r}^{\opt}}+\va{W}_{p}^{\opt}+\bar{\mathbf{a}}}   \notag \\
        &&+\min_{\probelement_{c} \in \probset_{c}}\nesp{\probelement_{c}}{\va{W}_{c}^{\opt}-\va{\price}\mathbf{y}^{\opt}+\bar{\mathbf{b}}}   \notag
        \eqfinv
        \\
        &\leq &\min_{\QQ \in \probset_{p} \cap \probset_{c}}\nesp{\QQ}{\va{W}_{c}^{\opt}+\va{W}_{p}^{\opt}}
        \eqfinv
    \end{eqnarray}
    and the social planner uses a risk measure with $\probset = \probset_{p} \cap \probset{c}$.
\end{remark}

The following proposition (Theorem 11 \citet{PhilpottFerrisWets16}) stands as a reverse statement for Proposition~\ref{prop_RASP_RAEQAD}. 
\begin{prop}\label{prop_article_Andy}
    Let the elements $x^{\opt}$, $\mathbf{x}_{r}^{\opt}$ and $\mathbf{y}_r^{\opt}$ be optimal solutions to $\RASP{\check{\FF}}$, with associated worst case probability measure $\va{\mu}$.  Then there exists prices $\va{\price}$ such that the couple $\np{\va{\price},\va{\mu}}$ forms a risk trading equilibrium for $\RAAD{\check{\FF}}$ with associated optimal solutions $\np{x^{\opt},\mathbf{x}_{r}^{\opt},\mathbf{y}^{\opt}}$. 
\end{prop}

Combining Proposition~\ref{prop_RASP_RAEQAD} and Proposition~\ref{prop_article_Andy}, we are able to state the following result of uniqueness of equilibrium.

\begin{cor}
     If both the producer's and consumer's criterion are strictly concave, and if each of the extreme points $\probelement_{k}$ charges all $\omega$, then $\RASP{\check{\FF}}$ admits a unique solution $(x^{\opt}, \mathbf{x}_{r}^{\opt}, \mathbf{y}^{\opt})$. Furthermore $\RAAD{\check{\FF}}$ admits unique optimal decisions $(x^{\opt}, \mathbf{x}_{r}^{\opt}, \mathbf{y}^{\opt})$. If, in addition, solving $\RASP{\check{\FF}}$ admit a unique worst case probability measure $\va{\mu}$, then equilibrium prices $\np{\va{\price},\va{\mu}}$ are unique.
\end{cor}
\begin{proof}
As each of the extreme points $\probelement_{k}$ charges all $\omega$, the risk averse social planner problem is strictly convex with linear constraints.
Thus there exists a unique solution $(x^{\opt}, \mathbf{x}_{r}^{\opt}, \mathbf{y}^{\opt})$ attained for a worst case probability $\va{\mu}$. 
Applying Proposition~\ref{prop_article_Andy}, we know that there exists $\va{\price}$ such that $\np{\va{\price},\va{\mu}}$ forms a risk trading equilibrium.
Suppose now that there exists two risk-trading equilibria $\np{\va{\price}^{1},\va{\mu}^{1}, x^{1},\mathbf{x}_{r}^{1}, \mathbf{y}^{1}}$ and $\np{\va{\price}^{2},\va{\mu}^{2}, x^{2},\mathbf{x}_{r}^{2}, \mathbf{y}^{2}}$. 
Then, by Proposition~\ref{prop_RASP_RAEQAD}, they both solve $\RASP{\check{\FF}}$ which admits a unique solution.
Consequently, we have 
$x^{\opt} = x^{1}= x^{2}$, $\mathbf{x}_{r}^{\opt} = \mathbf{x}_{r}^{1}=\mathbf{x}_{r}^{2}$ and $\mathbf{y}^{\opt} = \mathbf{y}^{1}=\mathbf{y}^{2}$. 

If in addition, $\va{\mu}^{1} = \va{\mu}^{2}$, then by Corollary~\ref{cor_uniqueness_risk_neutral}, we deduce that $\va{\price}^{1} = \va{\price}^{2}$ which ends the proof. 
\end{proof}
We have shown a first equivalence between $\RNSP{\prbt}$ and $\RNEQ{\prbt}$ and a second one between $\RASP{\check{\FF}}$ and $\RAAD{\check{\FF}}$. These equivalences lead to uniqueness of equilibrium if there is 
uniqueness of the solution of the social planner. A natural question arises: if $\RASP{\check{\FF}}$ has a unique solution, is there a unique equilibrium for $\RAEQ{\check{\FF}}$? The next section provides a 
simple counterexample.

\section{Multiple 
risk averse
equilibrium}\label{sec:multiple_eq}

In this section, we present a toy problem where $\RASP{\check{\FF}}$ has a unique optimum but there are three different equilibria for $\RAEQ{\check{\FF}}$. They are first found numerically using classical methods (PATH solver and a tâtonnement algorithm), then derived analytically. An interesting point is that the equilibrium found by PATH is unstable.

Let $\omeg = \na{1,2}$ and $\probset = \mathrm{conv}\ba{\np{\frac{1}{4},\frac{3}{4}},\np{
\frac{3}{4},\frac{1}{4}}}$. For simplicity of notation index by $i \in \na{1,2}$ the realization of each random variable.
We choose the following parameters: $V_{1}=4$, $V_{2}=\frac{48}{5}$, $c=\frac{23}{2}$, $c_{1}=1$, $c_{2}=\frac{7}{2}$, $r_{1}=2$, $r_{2}=10$.

\subsection{Multiple equilibrium}
\subsubsection{PATH solver}
First we look for equilibrium using GAMS with the solver PATH in the EMP framework (See
\citet{brook1988gams}, \citet{ferris2009extended} and 
\citet{ferris2000complementarity}). 
We have run GAMS from different starting points defined by a grid $100 \times 100$ over the square $\nc{1.220; 1.255} \times \nc{2.05; 2.18}$. We always find an equilibrium defined by
\begin{align*}
    \va{\price} = \np{\price_1,\price_2}
    = 
    \np{1.23578;2.10953}
    \eqfinv
\end{align*}
leading to risked adjusted welfare $ \np{2.134; 0.821}$ for producer and consumer respectively.

\subsubsection{Walras tâtonnement}
We now compute the equilibrium using a tâtonnement algorithm (See \citet{uzawa1960walras}).

\IncMargin{1em}
\begin{algorithm}
    \SetAlgoLined
    \KwData{MAX-ITER, $\np{\price_{1}^{0},\price_{2}^{0}},\tau$} 
    \For{$k$ from $0$ to MAX-ITER}{
        \emph{Compute an optimal decision for each player given a price} :\\
            $\qquad x, x_1,x_2 \in \argmax \bmes{}{\va{W}_{p}+\va{\price}\np{x+\mathbf{x}_{r}}}$\;
            $\qquad y_{1},y_{2} \in \argmax \nmes{}{\va{W}_{c}-\va{\price}\mathbf{y}}$\;
        \emph{Update the price} : \\
        $\qquad \price_{1} = \price_{1} - \tau\max\ba{0; y_{1}-\np{x+x_{1}}}$\;
        $\qquad \price_{2} = \price_{2} - \tau\max\ba{0; y_{2}-\np{x+x_{2}}}$\;
    }
    \KwRet{$\np{\price_{1},\price_{2}}$}
 \caption{Walras tâtonnement}
 \label{algo_walras}
\end{algorithm}
\DecMargin{1em}

Running algorithm~\ref{algo_walras} starting from $\np{1.25;2.06}$, respectively from $\np{1.22;2.18}$, with $100$ iterations and a step size of $0.1$, we find two new equilibria: 
\begin{align*}
    \va{\price}
    = 
    \np{1.2256;2.0698}
    \textrm{ and }
    \va{\price}
    = 
    \np{1.2478;2.1564}
    \eqfinv
\end{align*}
leading to risked-adjusted welfare for producer and consumer respectively $ \np{2.152; 0.798}$ and $\np{2.113; 0.845}$.  Notice that neither equilibrium dominates the other.

An alternative tatônnement method called FastMarket (see \citet{facchinei2007generalized}) finds the same equilibrium.

\subsection{Analytical results}
We now compute the three equilibrium analytically. Details of the computation are in~\ref{app:analytical}.

Consider two probabilities $\np{\underline{p}, 1-\underline{p}}$ and $\np{\bar{p}, 1-\bar{p}}$
Given prices $0 < \price_{1} < \price_{2}$, we solve the producer (resp. consumer) optimization problem.
Optimal decisions are derived in~\ref{appendix_optimal_control} and summed up in Table~\ref{tab:prod_solution} where $x_c$ is given by
\begin{equation*}
    x_{c}\np{\va{\price}} 
    = 
    \frac{1}{2\np{\price_{1} - \price_{2}}}
    \vardelim{
        \frac{\price_{2}^{2}}{c_{2}} 
        - \frac{\price_{1}^{2}}{c_{1}}
    } \eqfinp
\end{equation*}

\begin{table}
\centering
\begin{tabular}{| c | c | c c c |}
\hline
 & condition  & $ x^\sharp$ & $x_i^\sharp$ & $y_i^\sharp$  \\
   \hline case a) &
$x_c \leq \frac{\besp{\bar{p}}{\va{\price}}}{c}$ &
$\frac{\besp{\bar{p}}{\va{\price}}}{c}$&
$\frac{\price_i}{c_i}$ 
& $\frac{V_i - \price_i}{r_i}$\\
case b) &
$\frac{\besp{\bar{p}}{\va{\price}}}{c} \leq x_c \leq \frac{\besp{\underline{p}}{\va{\price}}}{c}$ &
 $x_c$ &
$\frac{\price_i}{c_i}$ 
& $\frac{V_i - \price_i}{r_i}$\\
case c) &
$\frac{\besp{\underline{p}}{\va{\price}}}{c} \leq x_c$&
$\frac{\besp{\underline{p}}{\va{\price}}}{c}$ &
$\frac{\price_i}{c_i}$ 
& $\frac{V_i - \price_i}{r_i}$\\
\hline
\end{tabular}
\caption{Optimal control for producer and consumer problems
\label{tab:prod_solution}}
\end{table}

We see that there are three regimes, depending only on the prices $(\price_1,\price_2)$,  of optimal first stage solutions. Case a) (resp. case c)), corresponds to a set of prices such
that $\EE_{\bar p}[\va W_p] < \EE_{\underline{p}}[\va W_p]$ (resp. $\EE_{\bar p}[\va W_p] > \EE_{\underline{p}}[\va W_p]$), and the optimal decision corresponds to an optimal risk-neutral decision with respect to one of the two extreme points of $\cQ$.
On the other hand, case b) corresponds to a set of prices such that the expected welfare is equivalent for all probability in $\cQ$, i.e. $\EE_{\bar p}[\va W_p] = \EE_{\underline{p}}[\va W_p]$.
In Figure~\ref{fig_analytical_solution}, the red area
corresponds to case a), the blue to case b) and the red to case c), separated by black lines of equations $\frac{\nesp{\bar{p}}{\va{\price}}}{c} = x_{c}\np{\va{\price}}$ and $\frac{\nesp{\underline{p}}{\va{\price}}}{c} = x_{c}\np{\va{\price}}$ respectively. 

We are now looking for prices $(\price_1,\price_2)$ such that the complementarity constraints are satisfied. 
For strictly positive prices, these constraints can be summed up as 
\begin{equation*}
    z_i\np{\va{\price}} =  x^{\opt}\np{\va{\price}} + x_i^{\opt}\np{\va{\price}} - y_i^{\opt}\np{\va{\price}} = 0 \eqsepv \qquad i \in \na{1,2}.
\end{equation*}

Accordingly we define excess supply functions $z^l_i$ for case $l \in \na{a,b,c}$, and $i \in \na{1,2}$. The red, blue and green lines corresponds to manifolds of null excess supply function for scenario $i$, that is of prices such that $z^l_i(\price_1,\price_2)=0$. When the lines cross we have $z^1_l = z^2_l = 0$, and thus we have candidate equilibrium. If the lines cross in the area of the same color we have an equilibrium. This is the case with the parameters chosen, and equilibrium can be derived in exact arithmetic.

We end with a few remarks derived from this example.

\begin{remark}
The PATH solver finds the blue equilibrium, Algorithm~\ref{algo_walras} finds the green and the red equilibrium as illustrated by Figure~\ref{fig_equil}. Interestingly it can be shown that the blue equilibrium is unstable in the sense that the dynamical system driven by $\price' = z(\price)$ is not locally stable (see~\cite{samuelson1941stability}) around the blue equilibrium (see~\ref{app:stability}).
\end{remark}

\begin{remark}
No equilibrium dominates another: if going from one equilibrium to another increases the (risk-adjusted) welfare of one agent, then it decreases the (risk-adjusted) welfare of the other.
\end{remark}

\begin{remark}
Using the analytical results we check that there exists a set of non-zero Lebesgue measure of parameters $V_{1}, V_{2}, c,c_{1},c_{2},r_{1}$, and $r_{2}$ (albeit small), that have three distinct equilibria with the same properties.
\end{remark}

\begin{remark}
We can show that the blue equilibrium is a convex combination of red and green equilibrium, illustrated on Figure~\ref{fig_analytical_solution} by the dashed blue line.
\end{remark}
\begin{figure}[ht]
    \centering
    \includegraphics[width=0.9\linewidth]{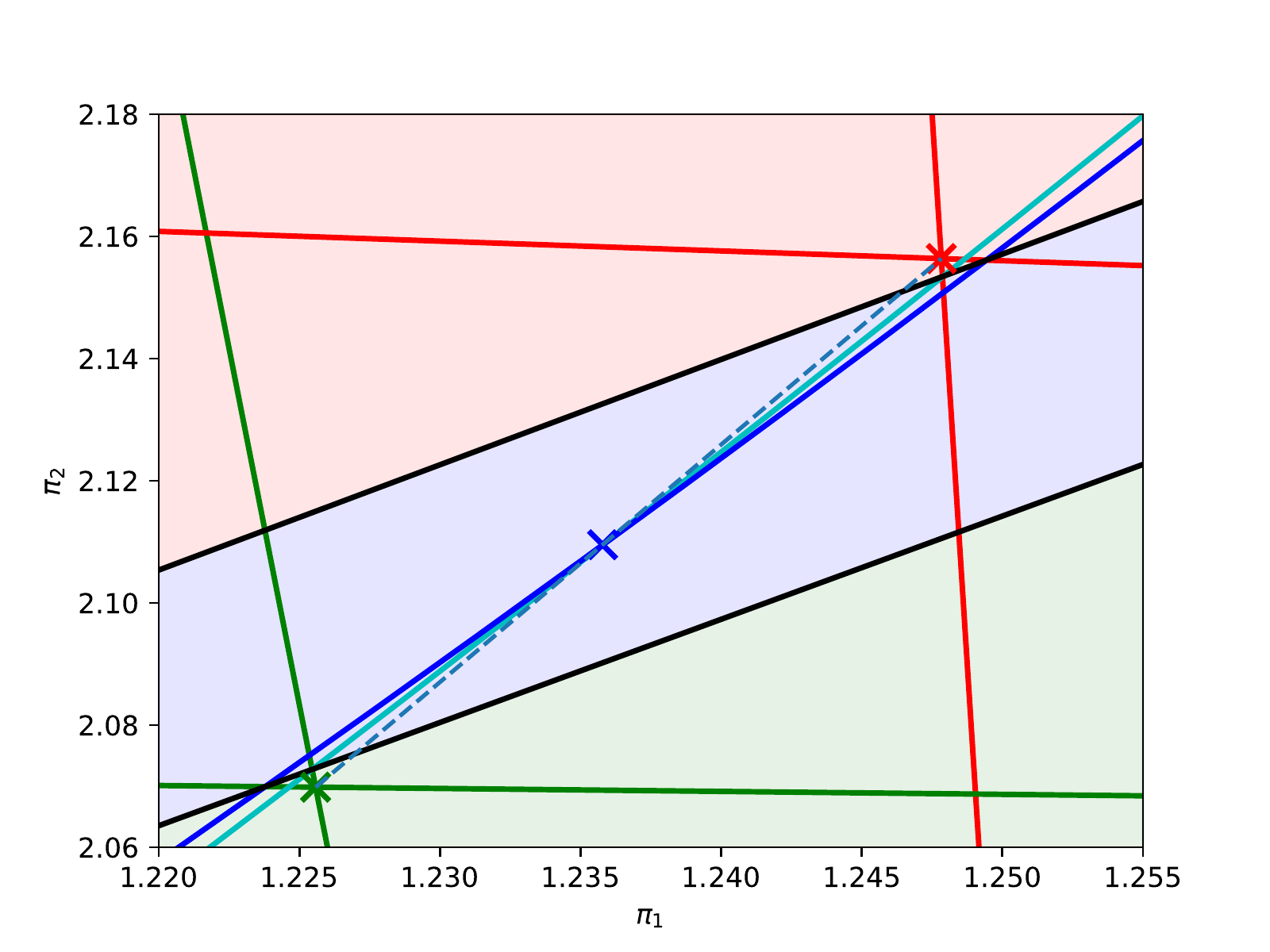}
    \caption{Null excess function per scenario manifold for
    $V_{1}=4$, $V_{2}=\frac{48}{5}$, $c=\frac{23}{2}$, $c_{1}=1$, $c_{2}=\frac{7}{2}$, $r_{1}=2$, $r_{2}=10$.}
    \label{fig_analytical_solution}
\end{figure}
\begin{figure*}[htbp]
\centering
\subfloat[around green equilibrium.\label{fig_equil_green}]{\includegraphics[width=0.3\textwidth]{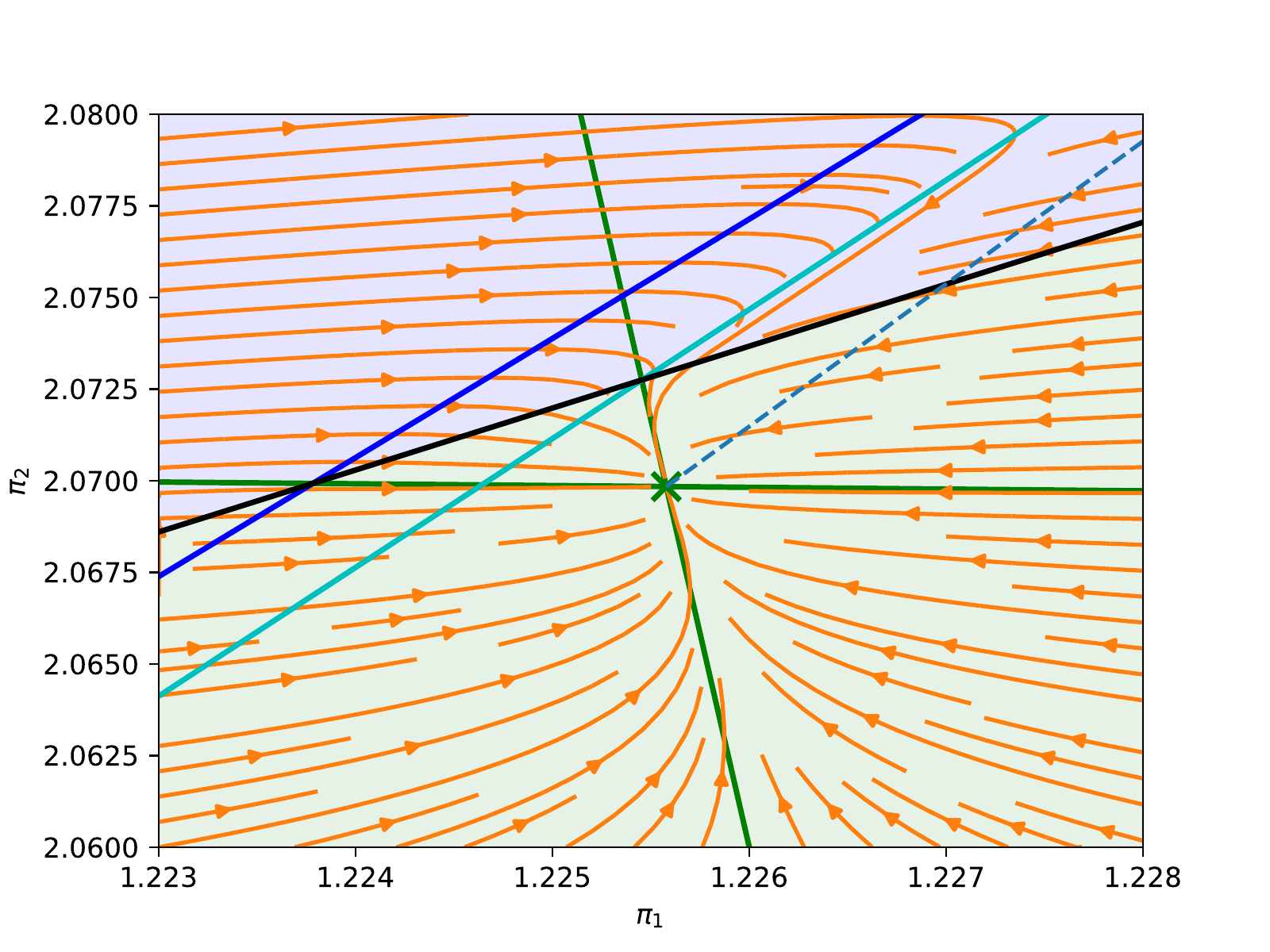}}\hfill
\subfloat[around blue equilibrium.\label{fig_equil_blue}] {\includegraphics[width=0.3\textwidth]{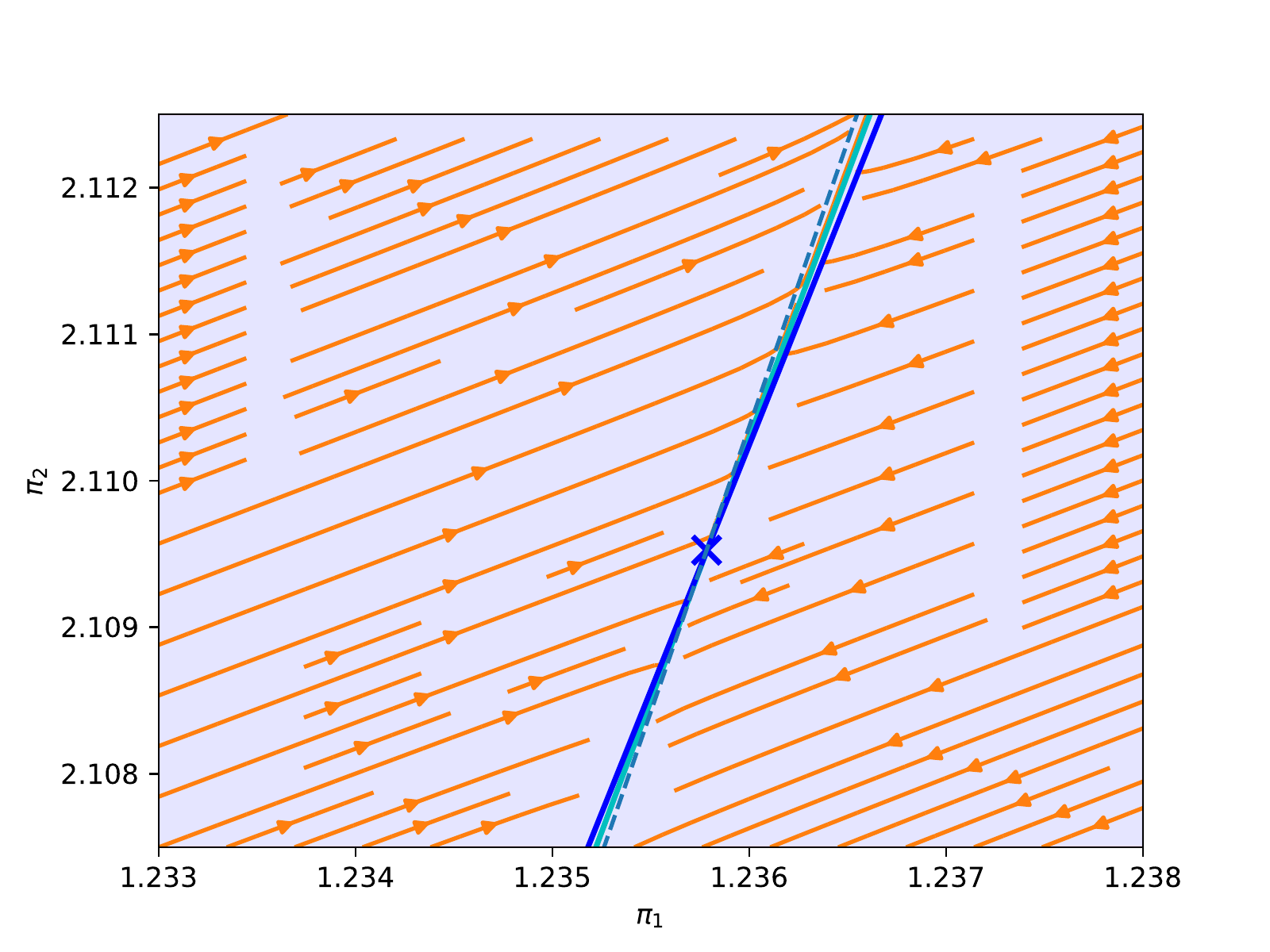}}\hfill
\subfloat[around red equilibrium.\label{fig_equil_red}]{\includegraphics[width=0.3\textwidth]{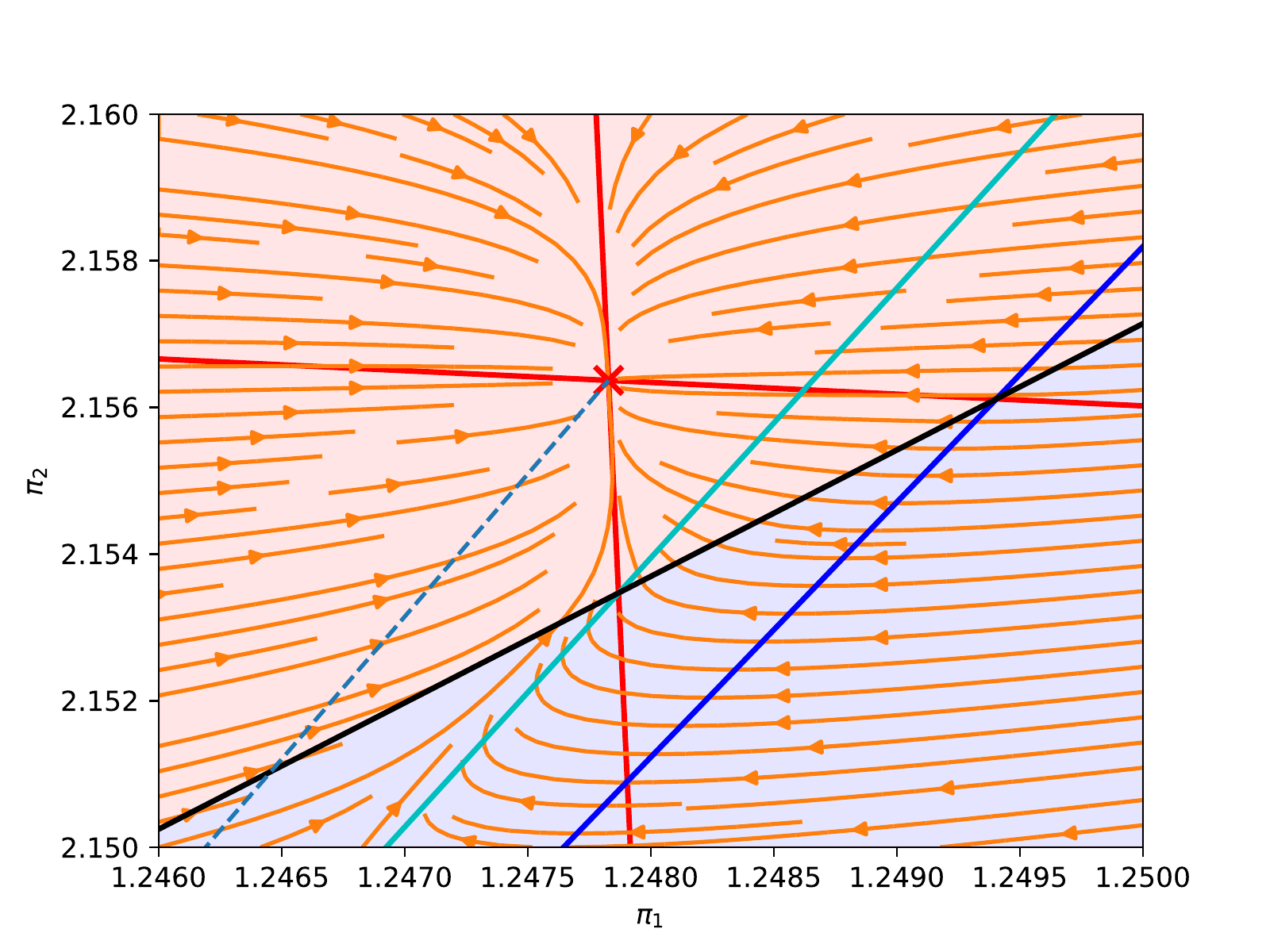}}
\caption{Representation of vector field $\mathbf{\price}' = z\np{\mathbf{\price}}$} \label{fig_equil}
\end{figure*}
\section*{Acknowledgments}
The first-named author want to thank French ambassy of New-Zealand for their administrative help and for the financial support thanks to France--New-Zealand friendship fund. 

The authors want to thank PGMO programs for their financial support.


\bibliography{biblio}

\appendix
\section{Analytical results}
\label{app:analytical}
We first analyses the best responses of the producer and the consumer given a price $\va{\price}$. Then, we deduce conditions on the price and find equilibrium prices.

\subsection{Parametric solution with respect to \texorpdfstring{$\mathbf{\price}$}{the price}}
Assume without loss of generality that $0<\price_1 < \price_2$.
\subsubsection{Statement of consumer's problem}
The consumer solves one problem per scenario $\omega_i$, $i=1,2$. 
Let $V_{1}$, $V_{2}$, $r_{1}$ and $r_{2}$ be strictly positive constants. 
The consumer problem for $\omega_i$ is 
\begin{equation*}
    \min_{y_{i}}
    \quad
    \price_{i}y_{i} -V_{i}y_{i} + \frac{1}{2}r_{i}y_{i}^{2} 
    \eqfinp
\end{equation*}

\subsubsection{Statement of producer's problem}
The  risk aversion of the producer is represented by a coherent risk measure $\mesr$ with risk set $\probset$.
Then the producer problem reads
\begin{equation*}
    \min_{x\geq 0, \mathbf{x}_{r} \geq 0} \quad 
    \FF \Bc{C\np{x} + \va{C}_{r}\np{\mathbf{x}_{r}} - \va{\price}(x+\mathbf{x}_{r})}
    \eqfinp
\end{equation*}
Note that in the case of two outcomes the probability $\PP$ measure can be defined by $\PP(\omega_1)$, which we denote $p$. Hence the probability set $\cP$ can be described by an 
interval $[\underline{p},\bar{p}]$.

Then the producer problem reads
\begin{align}\label{appendix_pb:producer-simple}
    \min_{x \geq 0, x_1 \geq 0, x_2 \geq 0} 
    \quad 
    \frac{1}{2}cx^2
    +
    \begin{multlined}[t]
        \max_{p \in [\underline{p},\bar{p}]}
        \Bga{
            p\Bp{\frac{c_{1} x_{1}^2}{2} - \price_1 \np{x+x_{1}}}
        \\
            + (1-p)\Bp{\frac{c_{2}x_{2}^2}{2} - \price_2 \np{x+x_{2}}}
        }
    \end{multlined}
\end{align}

\subsubsection{Statement of complementary constraints}
The complementary constraint states that a feasible solution is a solution where production is greater than demand for each scenario $\omega \in \omeg$. Moreover, we want equality 
between 
production and demand at equilibrium. These constraints are written
\begin{equation}
\label{appendix_cst:complementarity}
    0 \leq
    \np{x + \mathbf{x}_{r}\np{\omega} } -\mathbf{y}\np{\omega} 
    \perp
    \va{\price}\np{\omega}
    \geq 0
    \eqfinp
\end{equation}

\subsubsection{Analytic solution of the producer's problem}\label{appendix_optimal_control}
Focusing on the second stage problem of~\eqref{appendix_pb:producer-simple} we have
\begin{subequations}
\begin{align*}
    Q^{\np{\va{\price}}}\np{x} 
    &
    & = \max_{p \in [\underline{p},\bar{p}]} \quad
    \begin{multlined}[t]
        p\min_{ x_{1} \geq 0} \Ba{\frac{c_1x_{1}^2}{2} - \price_1 (x+x_{1})}
        \\
        +(1-p)\min_{ x_{2} \geq 0} \Ba{\frac{c_2x_{2}^2}{2} - \price_2 (x+x_{2})}
        \eqfinp
    \end{multlined}
\end{align*}
\end{subequations}

Note that for $i\in\na{1,2}$ $c_i >0$, hence we have 
$x_i^{\opt} = \frac{\price_i}{c_i} $
which in turn gives
\begin{align}
    Q^{\np{\va{\price}}}\np{x} & = \max_{p \in [\underline{p},\bar{p}]} \quad 
    - p\Bp{\frac{\price_1^2}{2c_1}+\price_1 x}
    - (1-p)\Bp{\frac{\price_2^2}{2c_2}+\price_2 x} \label{appendix_eq:Q_as_exp}\\
    &= \max_{p \in [\underline{p},\bar{p}]} \quad 
        p \Bp{\Bp{\frac{\price_2^2}{2c_2}- \frac{\price_1^2}{2c_1}}
        +\bp{\price_2  -\price_1 }x} 
        - \Bp{\frac{\price_2^2}{2c_2}+\price_2 x}
        \eqfinp
    \nonumber
\end{align}

Defining 
\begin{equation}
    x_{c}\np{\va{\price}}= \frac{-1}{\price_2 -\price_1} \Bp{\frac{\price_2^2}{2c_2}- \frac{\price_1^2}{2c_1}}
    \eqfinv
\end{equation}
we see that the worst case probability is given by 
\begin{equation*}
    p^{\opt}\np{\va{\price}} = 
    \begin{cases} 
        \bar{p} &\text{ if } x > x_{c}\np{\va{\price}} \eqfinv\\
        \underline{p} &\text{ if } x < x_{c}\np{\va{\price}} \eqfinv\\
        \text{ any } p \in [\underline{p},\bar{p}] & \text{ if } x = x_{c}\np{\va{\price}} \eqfinv
    \end{cases}
\end{equation*}
and thus Equation~\eqref{appendix_eq:Q_as_exp} yields
\begin{equation*}
    Q^{\np{\va{\price}}}\np{x} = 
    \begin{cases} 
        -\Besp{\bar{p}}{\frac{\va{\price}^2}{2\mathbf{c}_{r}}+ \va{\price}x}
        &\text{ if } x \geq x_{c}\np{\va{\price}}
    \eqfinv
    \\
        -\Besp{\underline{p}}{\frac{\va{\price}^2}{2\mathbf{c}_{r}}+ \va{\price}x}
        &\text{ if } x < x_{c}\np{\va{\price}}
        \eqfinp
    \end{cases}
\end{equation*}
Now the first stage problem (Problem~\eqref{appendix_pb:producer-simple}) reads  
\begin{equation*}
    \min_{x\geq 0} \quad 
    \frac{1}{2}cx^2  
    -\Besp{\bar{p}}{\frac{\va{\price}^2}{2\mathbf{c}_{r}}+ \va{\price}x} \1_{x \geq x_c}
    -\Besp{\underline{p}}{\frac{\va{\price}^2}{2\mathbf{c}_{r}}+ \va{\price}x} \1_{x < x_c}
    \eqfinp
\end{equation*}

We have 
\begin{equation*}
    \min_{x\geq x_c} \quad 
   \frac{1}{2}cx^{2} 
   +
   Q^{\np{\va{\price}}}\np{x}
    = 
    \begin{cases}
        -\frac{1}{2c}\besp{\bar{p}}{\va{\price}}^2  -\Besp{\bar{p}}{\frac{\va{\price}^2}{2\mathbf{c}_{r}}}
        &\text{ if } x_c \leq \frac{\besp{\bar{p}}{\va{\price}}}{c}
    \eqfinv
    \\
        \frac{1}{2}cx_c^2 -\Besp{\bar{p}}{\frac{\va{\price}^2}{2\mathbf{c}_{r}}+ \va{\price}x_c}
        &\text{ if }  \frac{\besp{\bar{p}}{\va{\price}}}{c} \leq x_c
    \eqfinv
    \end{cases}
\end{equation*}
attained at $\frac{\besp{\bar{p}}{\va{\price}}}{c}$ and $x_c$  respectively.

If $x_c > 0 $ we also have 
\begin{equation*}
    \min_{0 \leq x\leq x_c} \quad 
    \frac{1}{2}cx^{2} 
    +
    Q^{\np{\va{\price}}}\np{x}
    = 
    \begin{cases}
        \frac{1}{2}cx_c^2 -\Besp{\bar{p}}{\frac{\va{\price}^2}{2\mathbf{c}_{r}}+ \va{\price}x_c}
        &\text{ if } x_c \leq \frac{\besp{\underline{p}}{\va{\price}}}{c}
    \eqfinv
    \\
        -\frac{1}{2c}\besp{\underline{p}}{\va{\price}}^2  -\Besp{\underline{p}}{\frac{\va{\price}^2}{2\mathbf{c}_{r}}}
        &\text{ if }    \frac{\besp{\underline{p}}{\va{\price}}}{c} \leq x_c
    \eqfinv
    \end{cases}
\end{equation*}
attained at $x_c$ and $\frac{\besp{\underline{p}}{\va{\price}}}{c}$ respectively.
If $x_c \leq 0$ the solution given earlier holds.

Recall that $\besp{\underline{p}}{\va{\price}} \leq \besp{\overline{p}}{\va{\price}}$, thus  the optimal solution can be summed up in Table \ref{appendix_tab:prod_solution}

\begin{table}
\label{appendix_tab:simple-pb-sol}
\centering
\begin{tabular}{| c | c | c c c |}
\hline
 & condition  & $ x^\sharp$ & $x_i^\sharp$ & $y_i^\sharp$  \\
\hline case a) &
$x_c \leq \frac{\besp{\bar{p}}{\va{\price}}}{c}$ &
$\frac{\besp{\bar{p}}{\va{\price}}}{c}$&
$\frac{\price_i}{c_i}$ 
& $\frac{V_i - \price_i}{r_i}$
\\
case b) &
$\frac{\besp{\bar{p}}{\va{\price}}}{c} \leq x_c \leq \frac{\besp{\underline{p}}{\va{\price}}}{c}$ &
 $x_c$& 
$\frac{\price_i}{c_i}$ 
& $\frac{V_i - \price_i}{r_i}$
\\
case c) &
$\frac{\besp{\underline{p}}{\va{\price}}}{c} \leq x_c$&
$\frac{\besp{\underline{p}}{\va{\price}}}{c}$ &
$\frac{\price_i}{c_i}$ 
& $\frac{V_i - \price_i}{r_i}$\\
\hline
\end{tabular}
\caption{Optimal control for producer and consumer problems
\label{appendix_tab:prod_solution}}
\end{table}


\subsection{ Finding price equilibrium}\label{apppendix_analytical_results}
Looking at Table~\ref{appendix_tab:prod_solution} we see that there are three regimes, depending only on the prices $(\price_1,\price_2)$  of optimal first stage solutions. 
We are now looking for prices $(\price_1,\price_2)$ such that the complementarity constraint~\eqref{appendix_cst:complementarity} is satisfied. 
For strictly positive prices, this constraint can be summed up as 
\begin{equation}
    z_i\np{\va{\price}} = x^{\opt}\np{\va{\price}} + x_{1}^{\opt}\np{\va{\price}} - y_{i}^{\opt}\np{\va{\price}} =0 , \qquad i \in \na{1,2}.
\end{equation}

To go further we are going to split cases by defining the auxiliary excess demand function
\begin{align*}
    z^a_i\np{\va{\price}} & = 
\frac{\besp{\bar{p}}{\va{\price}}}{c} +
\frac{\price_i}{c_i} 
 - \frac{V_i - \price_i}{r_i}
 \eqfinv
 \\
 z^b_i\np{\va{\price}} & = 
x_{c}\np{\va{\price}} +
\frac{\price_i}{c_i} 
 - \frac{V_i - \price_i}{r_i}
 \eqfinv
 \\
 z^c_i\np{\va{\price}} & = 
\frac{\besp{\underline{p}}{\va{\price}}}{c} +
\frac{\price_i}{c_i} 
 - \frac{V_i - \price_i}{r_i}
 \eqfinv
\end{align*}
such that we have 
\begin{equation*}
    z = z^a \1_{cx_{c}\np{\va{\price}} \leq \besp{\bar{p}}{\va{\price}}}
    + z^b \1_{\besp{\bar{p}}{\va{\price}} \leq cx_{c}\np{\va{\price}} \leq \besp{\underline{p}}{\va{\price}}}
    + z^c \1_{\besp{\underline{p}}{\va{\price}} \leq cx_{c}\np{\va{\price}}  }
    \eqfinp
\end{equation*}

\subsubsection{Case a and c}

The set of prices such that $z_i^a(\price)=0$ are lines given by 
\begin{align*}
    \price_2 & = \frac{cc_1 V_1 - \bp{c_1 r_1 \bar{p} +c(r_1 + c_1)} \price_1 }{c_1r_1(1-\bar{p})} 
    \eqfinv
    \\
    \price_2 &= \frac{c c_2 V_2 - c_2 r_2 \bar{p}}{c_2 r_2 (1-\bar{p}) + c(r_2 + c_2)}
    \eqfinv
\end{align*}

and the equilibrium can be found by solving the linear system.
Case c is similar, subtituting $\bar{p}$ by $\underline{p}$. 

\subsubsection{Case b}
The set of prices such that $z_i^b(\price)=0$ are an ellipsoid and an hyperbola given by 
\begin{align*}
    \frac{1}{\price_1 - \price_2}\Bp{\frac{\price_2^2}{2c_2}- \frac{\price_1^2}{2c_1}} + \frac{\price_1}{c_1} - \frac{V_1 - \price_1}{r_1} 
    &= 0
    \eqfinv
    \\
    \frac{1}{\price_1 - \price_2}\Bp{\frac{\price_2^2}{2c_2}- \frac{\price_1^2}{2c_1}} + \frac{\price_2}{c_2} - \frac{V_2 - \price_2}{r_2} 
    &= 0
    \eqfinv
\end{align*}
whose affine equations read 
\begin{subequations}
    \begin{align*}
        \frac{\price_2^2}{2c_2} 
        - \Bp{\frac{1}{c_1}+\frac{1}{r_1}} \price_1 \price_2
        + \Bp{\frac{1}{r_1}+\frac{1}{2c_1}} \price_1^2 
        + (\price_2-\price_1)\frac{V_1}{r_1} 
        &= 0
        \eqfinv
        \\ 
         \Bp{\frac{1}{r_2}+\frac{1}{2c_2}} \price_2^2
         -  \Bp{\frac{1}{c_2}+\frac{1}{r_2}} \price_1 \price_2
          +\frac{1}{2c_1} \price_1^2
        - (\price_2-\price_1)\frac{V_2}{r_2} 
        &= 0
        \eqfinp
    \end{align*}
\end{subequations}

\section{Unstability of equilibrium}
\label{app:stability}
\begin{mydef}
Let $\va{\price}\np{t}$ be the general solution of the differential equation
\begin{align}\label{eq_app_system}
    \va{\price}' = z\np{\va{\price}}
    \eqfinv
\end{align}
such that $\va{\price}\np{0} = \va{\price}_{0}$
An equilibrium $\va{\price}^{\opt}$ such that $z\np{\va{\price}} = 0$ is said to be \emph{locally stable} if
for all $\epsilon > 0$, there exists $\delta >0$ such that
\begin{align}
    \nnorm{\va{\price}_{0} - \va{\price}^{\opt}} < \delta
    \Rightarrow
    \nnorm{\va{\price}\np{t} - \va{\price}^{\opt} }
    <
    \epsilon
    \eqsepv
    \forall t > 0
    \eqfinp
\end{align}
\end{mydef}
Using classical results from the field of Ordinary Differential Equations (see~\citet{mattheij2002ordinary}), the local stability can be determined from studying the linearization of the system around the equilibrium point.
\begin{prop}
Let $\va{\price}^{\opt}$ be an equilibrium point. Let $A$ be the Jacobian matrix of $z\np{\va{\price}}$ at point $\va{\price}^{\opt}$. Then $\va{\price}^{\opt}$ is stable if and only both real parts of eigenvalues of $A$ are strictly positive.
\end{prop}
    
Computing matrix $A$ and its eigenvalues  in exact arithmetic (using Maxima), we find that the blue equilibrium is unstable and that green and red equilibria are stable.
\end{document}

%% file: commandes.tex
\def\mathscr{\EuScript}

\newcommand{\probelement}{\QQ}
\newcommand{\probset}{\cQ}

\newcommand{\price}{\pi}

\newcommand{\cL}{\mathcal{L}}

\newcommand{\cP}{\mathcal{P}}
\newcommand{\cQ}{\mathcal{Q}}

\newcommand{\EE}{\mathbb{E}}
\newcommand{\FF}{\mathbb{F}}

\newcommand{\PP}{\mathbb{P}}
\newcommand{\QQ}{\mathbb{Q}}
\newcommand{\RR}{\mathbb{R}}

\newcommand{\opt}{^{\sharp}}                                

\newcommand{\np}[1]{(#1)}                                   
\newcommand{\bp}[1]{\big(#1\big)}                           
\newcommand{\Bp}[1]{\Big(#1\Big)}                           

\newcommand{\nc}[1]{[#1]}                                   
\newcommand{\bc}[1]{\big[#1\big]}                           
\newcommand{\Bc}[1]{\Big[#1\Big]}                           

\newcommand{\nce}[1]{[\![#1]\!]}                                       

\newcommand{\na}[1]{\{#1\}}                                 
\newcommand{\ba}[1]{\big\{#1\big\}}                         
\newcommand{\Ba}[1]{\Big\{#1\Big\}}                         
\newcommand{\Bga}[1]{\Bigg\{#1\Bigg\}}                      

\newcommand{\argmax}{\mathop{\arg\max}}                     

\newcommand{\omeg}{\Omega}                                  
\newcommand{\prbt}{\mathbb{P}}                              
\newcommand{\espe}{\mathbb{E}}                              
\newcommand{\mesr}{\mathbb{F}}                              

\makeatletter
\def\va@a{\boldsymbol{\va@arg^{\textstyle\text{\unboldmath$\scriptstyle\va@expo$}}_{\textstyle\text{\unboldmath$\scriptstyle\va@index$}}}}
\def\va#1{\def\va@expo{}\def\va@index{}\def\va@arg{\uppercase{#1}}%
  \@ifnextchar^{\va@h}{\@ifnextchar_\va@u\va@a}}
\def\va@h^#1{\def\va@expo{#1}\@ifnextchar_\va@hu\va@a}
\def\va@u_#1{\def\va@index{#1}\@ifnextchar^\va@uh\va@a}
\def\va@hu_#1{\def\va@index{#1}\va@a}
\def\va@uh^#1{\def\va@expo{#1}\va@a}
\makeatother

\newcommand{\normdelim}[1]{\nc{#1}}                         
\newcommand{\bigdelim}[1]{\bc{#1}}                          
\newcommand{\Bigdelim}[1]{\Bc{#1}}                          
\newcommand{\vardelim}[1]{\left(#1\right)}                  

\newcommand{\nesp}[2]{\espe_{#1}\normdelim{#2}}           
\newcommand{\besp}[2]{\espe_{#1}\bigdelim{#2}}            
\newcommand{\Besp}[2]{\espe_{#1}\Bigdelim{#2}}            

\newcommand{\nmes}[2]{\mesr_{#1} \normdelim{#2}}           
\newcommand{\bmes}[2]{\mesr_{#1} \bigdelim{#2}}            
\newcommand{\Bmes}[2]{\mesr_{#1} \Bigdelim{#2}}            

\newcommand{\nnorm}[1]{\|#1\|}                              

\def\eqsepv{\; , \enspace}                                  
\def\eqfinv{\; ,}                                           
\def\eqfinp{\; .}                                           

\newcommand{\finpreuvesymb}{$\Box$}
\newcommand{\finremarksymb}{$\Diamond$}
\newcommand{\finexemplesymb}{$\triangle$}
\newcommand{\finpreuve}{\ \hspace*{\fill}\finpreuvesymb}
\newcommand{\finremark}{\ \hspace*{\fill}\finremarksymb}
\newcommand{\finexemple}{\ \hspace*{\fill}\finexemplesymb}

\makeatletter
\def\endproof{\finpreuve\@endtheorem}
\def\endremark{\finremark\@endtheorem}
\def\endexample{\finexemple\@endtheorem}
\makeatother